\documentclass[11pt]{article}

\usepackage{fullpage}
\usepackage{amssymb}
\usepackage{amsfonts}
\usepackage{amsmath}
\usepackage{amsthm}
\usepackage{latexsym}
\usepackage{epsfig}
\usepackage{indentfirst}
\usepackage{makeidx}
\usepackage{hyperref}
\usepackage{pdfpages}
\usepackage{graphicx}
\usepackage{mathtools}

\usepackage[numbers]{natbib} 

\theoremstyle{plain} %% This is the default
\newtheorem{thm}{Theorem}[section]

\newtheorem{cor}[thm]{Corollary}
\newtheorem{lem}[thm]{Lemma}
\newtheorem{prop}[thm]{Proposition}

\theoremstyle{definition}
\newtheorem{defn}[thm]{Definition}

\newtheorem{example}[thm]{Example}

\newtheorem{remark}[thm]{Remark}

\theoremstyle{remark}

\newcommand{\nn}{\mathbb{N}}

\newcommand{\ip}[1]{\:\quotient{#1}{\cong}}

\makeatletter
\newcommand*{\@old@slash}{}\let\@old@slash\slash
\def\slash{\relax\ifmmode\delimiter"502F30E\mathopen{}\else\@old@slash\fi}
\makeatother
\def\backslash{\delimiter"526E30F\mathopen{}}
\newcommand{\quotient}[2]{#1\slash#2}
\newcommand{\card}[1]{\ensuremath{\left|#1\right|}}
\newcommand{\mcc}{\mathcal{C}}

\newcommand{\mcf}{\mathcal{F}}
\newcommand{\mcg}{\mathcal{G}}

\newcommand{\eqdef}{\stackrel{\rm def}{=}}

\begin{document}

% \removequeries

\title{An algebraic formulation of the graph reconstruction conjecture\\~\\}

\author{ Igor C. Oliveira\footnote{Supported in part by NSF grants CCF-0915929 and CCF-1115703.}\\
  \small{Columbia University}\\
  \small{\texttt{oliveira@cs.columbia.edu}} \and Bhalchandra
  D. Thatte\footnote{Supported by CNPq grant 151782/2010-5 and by
    MaCLinC Project at Universidade de S\~{a}o Paulo.}\\
  \small{Universidade de S\~{a}o Paulo}\\
  \texttt{\small{thatte@ufmg.br}}\\~\\
}

%\date{}

\maketitle

\begin{abstract}

  The graph reconstruction conjecture asserts that every finite simple
  graph on at least three vertices can be
  reconstructed up to isomorphism from its deck - the collection of its
  vertex-deleted subgraphs. Kocay's Lemma is an important tool in graph
  reconstruction.  Roughly speaking, given the deck of a graph $G$ and
  any finite sequence of graphs, it gives a linear constraint that every
  reconstruction of $G$ must satisfy.

  Let $\psi(n)$ be the number of distinct (mutually non-isomorphic)
  graphs on $n$ vertices, and let $d(n)$ be the number of distinct decks
  that can be constructed from these graphs. Then the difference
  $\psi(n) - d(n)$ measures how many graphs cannot be
  reconstructed from their decks. In particular, the graph
  reconstruction conjecture is true for $n$-vertex graphs if and only if
  $\psi(n) = d(n)$. 
  
  We give a framework based on Kocay's lemma to study
  this discrepancy. We prove that if $M$ is a matrix of covering numbers
  of graphs by sequences of graphs, then $d(n) \geq
  \mathsf{rank}_\mathbb{R}(M)$. In particular, all $n$-vertex graphs are
  reconstructible if one such matrix has rank $\psi(n)$. To complement
  this result, we prove that it is possible to choose a family of
  sequences of graphs such that the corresponding matrix $M$ of covering
  numbers satisfies $d(n) = \mathsf{rank}_\mathbb{R}(M)$.
\end{abstract}

\section{Introduction}

The graph reconstruction conjecture was proposed by \citet{ulam1960} and
\citet{kelly1942}. Informally, it states that if two finite, undirected,
simple graphs on at least three vertices have the same collection
(\emph{multi-set} or \emph{deck}) of \emph{unlabelled} vertex-deleted
subgraphs, then the graphs are isomorphic; in other words, any such
graph can be \emph{reconstructed} up to isomorphism from the collection
of its unlabelled vertex-deleted subgraphs.

The conjecture has been verified by \citet{mckay1997small} for all
undirected, finite, simple graphs on eleven or fewer vertices. In
addition, it has been proven for many particular classes of graphs, such
as regular graphs, disconnected graphs and trees
(\citet{kelly1957congruence}). In fact, \citet{bollob'as-1990} showed that for almost all graphs, just three (carefully chosen)
subgraphs in the deck are sufficient to reconstruct the graph. On the other hand, a similar
conjecture does not hold for directed graphs:
\citet{stockmeyer-1977,stockmeyer-1981} constructed a number of infinite
families of non-reconstructible directed graphs. For a more
comprehensive introduction to the problem, we refer to a survey by
\citet{bondy-1991}. For the standard graph theoretic terminology not
defined here, we refer to \citet{west2001}.

Kelly's Lemma \cite{kelly1957congruence} is one of the most useful
results in graph reconstruction. Let $s(F,G)$ denote the number of
subgraphs of $G$ isomorphic to $F$. Kelly's lemma states that for $v(F)
< v(G)$, the parameter $s(F,G)$ is reconstructible, in the sense that if
$G$ and $G^\prime$ have the same deck then $s(F,G') = s(F,G)$. Several
propositions in graph reconstruction rely on this useful lemma.

Kocay's Lemma \cite{kocay-1981b} allows us, to some extent, to overcome
the restriction $v(F) < v(G)$ in Kelly's lemma. It provides a linear
constraint on $s(\cdot,G)$ that must be satisfied by every
reconstruction of $G$. Informally, it says that, if $\mathcal{F} = (F_1,
\ldots, F_m)$ is a sequence of graphs, each of which has at most
$v(G)-1$ vertices, then there are constants $c(\mathcal{F},H)$ such that
the value of the sum $\sum_H c(\mathcal{F},H) \cdot s(H,G)$ is
reconstructible, where the sum is taken over all unlabelled $n$-vertex
graphs $H$.  Roughly speaking, the constant $c(\mathcal{F},H)$ counts
the number of ways to \emph{cover} the graph $H$ by graphs in the
sequence $\mathcal{F}$.

Kocay's Lemma has been used to show several interesting results in graph
reconstruction. For instance, by carefully selecting the sequence
$\mathcal{F}$, it is possible to give a simple proof that disconnected
graphs are reconstructible. In addition, it can be used to show that the
number of perfect matchings, the number of spanning trees, the
characteristic polynomial, the chromatic polynomial, and many other
parameters of interest are reconstructible; see \citet{bondy-1991}.

It is natural to wonder whether even more restrictions may be imposed on
the reconstructions of $G$ by applications of Kocay's Lemma. Recall that
it is possible to use different sequences of graphs in each invocation
of the lemma, and as explained before, for each sequence we get a linear
constraint that the reconstructions of $G$ must satisfy. By analysing
such equations one would expect to obtain a wealth of information about
the structure of any reconstruction of $G$ (perhaps enough equations may
even allow us to conclude that $G$ is reconstructible). In this paper we
investigate how much information one can obtain by setting up such
equations.

We prove that the equations obtained by applying Kocay's Lemma to the
deck of a graph $G$ using distinct sequences of graphs provide important
information not only about the reconstructions of $G$, but also on the
total number of non-reconstructible graphs on $n$ vertices. More
formally, let $d(n)$ be the number of distinct decks obtained from
$n$-vertex graphs. We show that if $M$ is the matrix of coefficients
corresponding to these equations, then $d(n) \geq
\mathsf{rank}_\mathbb{R}(M)$, i.e., the rank of this matrix provides a
lower bound on the number of distinct decks. In particular, the
existence of a full-rank matrix of coefficients would imply that all
graphs on $n$ vertices are reconstructible. In addition, we give a proof
that there exist $d(n)$ sequences of graphs $\mcf_1, \mcf_2,\ldots,
\mcf_{d(n)}$, with corresponding matrix $M$ of covering numbers, such
that $\mathsf{rank}_\mathbb{R}(M) = d(n)$. In other words, if the graph
reconstruction conjecture holds for graphs with $n$ vertices, then there
is a corresponding full-rank matrix certifying this statement.

We state our results in more generality for graphs, hypergraphs,
directed graphs, and also for classes of graphs for which similar
equations can be constructed; for example, analogous results hold for
planar graphs, disconnected graphs and trees.

Similar system of equations where considered by Kocay \citep{kocay1982},
where he restricted the total number of edges appearing in each sequence
of graphs on a given system of equations to be the same. Interestingly,
in this case it is not possible to show an equivalence to the graph
reconstruction conjecture. In particular, Kocay computed the ratio of the
  number of independent edge-identities and the number of mutually
  non-isomorphic graphs with $v$ vertices, $e$ edges, and no
  isolated vertices (for small parameters $v,e$), and
observed that these values can be strictly less than
one. Kocay asked if the reconstruction conjecture would
  fail to be true if the ratio became small enough.

Our contribution may be summarised as follows. We remove the restriction,
  as in Kocay's paper \citep[Theorem 5.2 in][]{kocay-1981b}, that the
  total number of edges be fixed among the sequences of graphs used to derive
  edge-identities. We show that the number of independent
  equations available at our disposal is precisely the number of
  distinct decks on a given number of vertices. Thus we give an
  algebraic characterisation, based on Kocay's lemma, for the
  discrepancy between the number of different decks and the number of
  distinct graphs - a measure of how badly Ulam's conjecture would fail
  to hold, if indeed it were to be false. In view of the result of Bollob{\'a}s
  mentioned earlier, the ratio of the number of independent equations
  and the number of distinct graphs cannot be small.

A different mathematical perspective on such equations is presented in
Mnukhin \citep{mnukhin1992}, where reconstruction problems are discussed
in the more general context of orbit algebras. Mnukhin's paper also
mentions a formulation of Ulam's conjecture in algebraic terms, based on
whether the graph algebra is generated by disconnected graphs
only. While there may be a
 translation between the two formulations, this is not immediately
obvious to the authors. We refer the reader to Mnukhin's paper for further details,
and to the original reference \citep{mnukhin1982reconstruction} (in
Russian) discussed in \citep{mnukhin1992}. Our
  results are simple to prove, can be specialised to several classes of
  graphs, as well as generalised to digraphs and hypergraphs, and
  provide an exact characterisation of the maximum number of independent
  equations.

Finally, our results may also be viewed as a limitation of the
lemmas of Kelly and Kocay (which is proved using Kelly's lemma), and in
this regard we share the pessimism expressed by Tutte (see Chapter 9, page 113, \citep{tutte1998graph}). The fact that the number of independent equations
is equal to the number of decks suggests that the
difficulties with Ulam's conjecture lie somewhere else. In particular, it seems
unlikely that applications of Kelly's lemma and Kocay's lemma
will shed light on these difficulties.

\section{Preliminaries}

In this paper, we consider general finite graphs - undirected graphs,
directed graphs, hypergraphs, graphs with or without multiple edges, and
with or without loops. We take the vertex set of a graph to be a finite
subset of $\nn$. We write $V^{(k)}$ for the family of $k$-element
subsets of a set $V$. Further, we use the notation $v(G) \eqdef |V(G)|$
    and $e(G) \eqdef |E(G)|$.

\begin{defn}[Graphs]
  A \emph{hypergraph} $G$ is a triple $(V,E,\phi)$, where $V$ is its
  \emph{vertex set} (also called \emph{ground set}, and written as
  $V(G)$) and $E$ is its set of \emph{hyperedges} (written as $E(G)$),
  and a map $\phi:E \rightarrow 2^V\backslash
  \emptyset$. An \emph{undirected graph} $G$ is a
  hypergraph with the restriction that $\phi:E \rightarrow V^{(1)}\cup
  V^{(2)}$; in this case we call a hyperedge $e$ an \emph{edge} (if
  $\card{\phi(e)} = 2$) or a \emph{loop} (if $\card{\phi(e)} = 1$). An
  undirected graph is \emph{simple} if it contains no loop. A
  \emph{directed graph} $G$ is a triple $(V,E,\psi)$, where $V$ is its
  vertex set and $E$ is the set of its \emph{arcs}, and a map $\psi:E
  \rightarrow V\times V$. The first element of $\psi(e)$ is called the
  \emph{tail} of the arc $e$, and the second element of $\psi(e)$ is
  called the \emph{head} of $e$. We denote the set of all finite graphs
  (including hypergraphs, undirected graphs and directed graphs) by
  $\mcg^*$.\footnote{Observe that we are defining graphs using triples
    because multiple edges are allowed.}
\end{defn}

\begin{remark}
  Although our results and proofs are stated in full generality, it may
  be helpful in a first reading to consider only finite, simple,
  undirected graphs.
\end{remark}

\begin{defn}[Graph isomorphism]
  Let $G$ and $H$ be two graphs. We say that $G$ and $H$ are
  \emph{isomorphic} (written as $G\cong H$) if there are one-one maps
  $f:V(G) \rightarrow V(H)$ and $g:E(G) \rightarrow E(H)$ such that an
  edge $e$ and a vertex $v$ are incident in $G$ if and only the edge
  $g(e)$ and the vertex $f(v)$ are incident in $H$. Additionally, in the
  case of directed graphs, a vertex $v$ is the head (or the tail) of an
  arc $e$ if and only if $f(v)$ is the head (or, respectively, the tail)
  of $g(e)$. The isomorphism class of a graph $G$, denoted by $\ip{G}$,
  is the set of graphs isomorphic to $G$.
\end{defn}

\begin{defn}
  A \emph{class of graphs} is a set of graphs that is closed under
  isomorphism. A class of graphs is said to be \emph{finite} if contains
  finitely many isomorphism classes.
\end{defn}

\begin{defn}[Reconstruction]
  Let $G$ be graph and let $v$ be a vertex of $G$. The induced subgraph
  of $G$ obtained by deleting $v$ and all edges incident with $v$ is
  called a \emph{vertex-deleted subgraph} of $G$, and is written as
  $G-v$. We say that $H$ is a \emph{reconstruction} of $G$ (written as
  $H \sim G$) if there is a one-one map $f:V(G) \rightarrow V(H)$ such
  that for all $v \in V(G)$, the graphs $G-v$ and $H-f(v)$ are
  isomorphic. The relation $\sim$ is an equivalence relation.
  %% DO NOT DELETE - decks may be useful somewhere else.
  % The \emph{deck} of a graph $G \in \overline{\mcg}$ (written as
  % $d_G$) is a map $d_G:\overline{\mcg} \rightarrow \nn$ such that
  % $d_G(F)$ is the number of vertex-deleted subgraphs of $G$ that are
  % isomorphic to $F$. Given graphs $G$ and $H$, we have $d_G = d_H$ if
  % and only if $G \sim H$.
  % We denote the set of reconstructions of $G$ by $[G]$.
  We say that a graph $G$ is \emph{reconstructible} if every
  reconstruction of $G$ is isomorphic to $G$ (i.e., if $H\sim G$ implies
  $H\cong G$). A parameter $t(G)$ is said to be \emph{reconstructible}
  if $t(H) = t(G)$ for all reconstructions $H$ of $G$. Let $\mcc$ be a
  class of graphs. We say that $\mcc$ is \emph{recognisable} if, for any
  $G \in \mcc$, every reconstruction of $G$ is in $\mcc$. Furthermore,
  we say that $\mcc$ is \emph{reconstructible} if every graph $G \in
  \mcc$ is reconstructible.
\end{defn}

%% DO NOT DELETE - decks may be useful later
% Informally speaking, a deck of a graph may be interpreted as the
% collection (multi-set) of its unlabelled vertex-deleted subgraphs by
% restricting the map $d_G$ to those $F$ in $\overline{\mcg}$ for which
% $d_G(F)$ is non-zero and observing that the map $d_G$ is constant over
% each isomorphism class of $\overline{\mcg}$; so $d_G$ may be thought
% of as a map from the set of isomorphism classes of $\overline{\mcg}$
% to $\nn$.

\begin{example}
  \label{ex:pathological}
  Let $G(V,E,\phi)$ be a hypergraph. The number of edges incident with
  all vertices (i.e., edges $e\in E$ such that $\phi(e) = V$, which we
  call \emph{big edges}), is not a reconstructible parameter. For
  example, if $G^k$ is a graph obtained from $G$ by adding $k$ new edges
  $e_1,e_2,\dots,e_k$ and making them incident with all vertices in $V$,
  then $G^k$ is a reconstruction of $G$. In this sense, no hypergraphs
  are reconstructible, and each hypergraph has infinitely many mutually
  non-isomorphic reconstructions. If $G$ is a graph in class $\mcc$,
  then $\mcc$ is not recognisable if for some $k$, the graph $G^k$ is
  not in $\mcc$; and $\mcc$ is not finite if graphs $G^k$ are all in
  $\mcc$. On the other hand, the number of \emph{small edges}, i.e.,
  edges $e\in E$ such that $\phi(e) \neq V$, is a reconstructible
  parameter.
\end{example}

In view of the above example, we will always use $\mcg^*$ for the set of
all graphs, $\mcg$ for the set of all graphs without big edges, and
$\mcg_n$ for the set of $n$-vertex graphs without big edges.  A class
$\mcc_n$ will always be a subset of $\mcg_n$. We will use the following
slightly restrictive definitions for some other reconstruction terms.
% The deck of a graph $G$ will still be a map
% $d_G:\overline{\mcg}\rightarrow \nn$,

\begin{defn}
  A graph $G$ in $\mcg$ is reconstructible if it is
  \emph{reconstructible modulo big edges}, i.e., if $G^\prime$ is a
  reconstruction of $G$ and $G^\prime \in \mcg$, then $G^\prime$ is
  isomorphic to $G$. A subclass $\mcc$ of $ \mcg$ is recognisable if for
  each graph $G$ in $\mcc$, each reconstruction of $G$ in $\mcg$ is also
  in $\mcc$. A subclass $\mcc$ of $\mcg$ is reconstructible if each
  graph in $\mcc$ is reconstructible (modulo big edges).
\end{defn}

\begin{example}
  Disconnected undirected graphs on 3 or more vertices are recognisable
  and reconstructible. However, there are classes of graphs that are
  recognisable, but not known to be reconstructible. An important
  example is the class of planar graphs
  (\citet{bilinski2007reconstruction}).
\end{example}

Since $\cong$ and $\sim$ are equivalence relations, the \emph{quotient}
notation may be conveniently used to define various equivalence classes
of graphs. We write the set of all isomorphism classes of graphs as
$\quotient{\mcg^*}{\cong}$; analogously we use
$\quotient{\mcg_n}{\cong}$, $\quotient{\mcc}{\cong}$,
$\quotient{\mcc_n}{\cong}$, and so on. We define an \emph{unlabelled
  graph} to be an isomorphism class of graphs. But sometimes we abuse
the notation slightly, e.g., if a quantity is invariant over an
isomorphism class $H$, then in the same context we may also use $H$ to
mean a representative graph in the class. Similarly, we denote various
reconstruction classes by $\quotient{\mcg}{\sim}$,
$\quotient{\mcg_n}{\sim}$, $\quotient{\mcc}{\sim}$,
$\quotient{\mcc_n}{\sim}$, and so on. Note that equivalence classes of
any class of graphs under $\sim$ are refined by $\cong$; in particular,
$\card{\quotient{\mcc_n}{\sim}} \leq \card{\quotient{\mcc_n}{\cong}}$,
and equality holds if and only if the class $\mcc_n$ is
reconstructible. We will refer to reconstruction classes of $\mcc_n$
(i.e., members of $\quotient{\mcc_n}{\sim}$) by $R_1,R_2, \dots$, and
isomorphism classes of $R_i$ (i.e., members of $\quotient{R_i}{\cong}$)
by $R_{i,1}, R_{i,2}, \dots$.

Given graphs $G$ and $H$, the number of subgraphs of $G$ isomorphic to
$H$ is denoted by $s(H,G)$. The following two subgraph counting lemmas
are important results about the reconstructibility of the parameter
$s(H,G)$.

\begin{lem}[Kelly's Lemma, \cite{kelly1957congruence}]
  \label{l:kelly}
  Let $H$ be a reconstruction of $G$. If $F$ is any graph such that
  $v(F) < v(G)$, then $s(F,G) = s(F,H)$.
\end{lem}

% \begin{proof}
%   Note that each subgraph of $G$ isomorphic to $F$ occurs in exactly
%   $v(G) - v(F)$ of the vertex-deleted subgraphs $G - v$. It follows that
%   \begin{equation}\nonumber
%     s(F,G) = \frac{1}{v(G) - v(F)} \sum_{v \in V(G)} s(F,G-v).
%   \end{equation}
%   Since the right-hand side of this formula depends only on $d_G$, the
%   parameter $s(F,G)$ is reconstructible.
% \end{proof}

\begin{defn}
  Let $G$ be a graph and let $\mcf \coloneqq (F_1,F_2, \ldots , F_m)$ be
  a sequence of graphs. A \emph{cover} of $G$ by $\mcf$ is a sequence
  $(G_1,G_2, \ldots , G_m)$ of subgraphs of $G$ such that $G_i \cong
  F_i$, $1 \leq i \leq m$, and $\bigcup G_i = G$. The number of covers
  of $G$ by $\mcf$ is denoted by $c(\mcf,G)$.
\end{defn}

\begin{lem}[Kocay's Lemma, \cite{kocay-1981b}]
  \label{l:kocay}
  Let $G$ be a graph on $n$ vertices. For any sequence of graphs $\mcf
  \coloneqq (F_1,F_2, \ldots , F_m)$, where $v(F_i) < n$, $1 \leq i \leq
  m$, the parameter
  \begin{equation} 
    \nonumber \sum_{H} c(\mcf,H)s(H,G)\nonumber
  \end{equation}
  is reconstructible, where the sum is over all unlabelled $n$-vertex
  graphs $H$.
\end{lem}

\begin{proof}
  We count in two ways the number of sequences $(G_1, \ldots, G_m)$ of
  subgraphs of $G$ such that $G_i \cong F_i$, $1 \leq i \leq m$. We have
  \begin{equation} \label{eq1} \prod_{i=1}^m s(F_i,G) = \sum_X c(\mcf,X)
    s(X,G),
  \end{equation}
  where the sum extends over all unlabelled graphs $X$ on at most $n$
  vertices. Since $v(F_i) < n$, it follows by Kelly's Lemma that the
  left-hand side of this equation is reconstructible. On the other hand,
  the terms $c(\mcf,X) s(X,G)$ are also reconstructible whenever $v(X) <
  n$. The result follows after rearranging Equation \ref{eq1}.
\end{proof}

To state our results in full generality, we make the following
definition.

\begin{defn}
  Let $\mcc_n$ be a class of graphs on $n$ vertices. We say that
  $\mcc_n$ \emph{satisfies Kocay's lemma} if, for every graph $G \in
  \mcc_n$ and every sequence of graphs $\mcf = (F_1, F_2, \ldots, F_m)$,
  where $v(F_i) < n$, $1 \leq i \leq m$, the sum
  \begin{equation} \nonumber \sum_{H \in \ip{\mcc_n}}
    c(\mcf,H)s(H,G)\nonumber
  \end{equation}
  is reconstructible.
\end{defn}

The following proposition gives a simple condition that is sufficient
for a class of graphs $\mcc_n$ to satisfy Kocay's lemma.

\begin{prop}
  Let $\mcc_n$ be a class of graphs on $n$ vertices. Suppose that
  $s(H,G)$ is reconstructible for every $G \in \mcc_n$ and for every
  $n$-vertex graph $H \notin \mcc_n$. Then the class $\mcc_n$ satisfies
  Kocay's lemma.
\end{prop}

\begin{proof}
  Let $G \in \mcc_n$. Let $\mcf \coloneqq (F_1, F_2, \ldots, F_m)$ be
  any sequence of graphs such that $v(F_i) < n$, $1 \leq i \leq m$. We
  write the R.H.S. of Equation \ref{eq1} as
  \[
  \sum_{H \in \ip{\mcc_n}} c(\mcf,H)s(H,G) + \sum_{H \notin \ip{\mcc_n}}
  c(\mcf,H)s(H,G),
  \]
  where the second summation is reconstructible. Now we rearrange the
  terms in Equation \ref{eq1} to obtain $\sum_{H \in \ip{\mcc_n}}
  c(\mcf,H)s(H,G)$.
\end{proof}

The class of connected simple graphs satisfies Kocay's lemma, since if
$G$ is any connected graph and $H$ is any disconnected graph, then
$s(H,G)$ is reconstructible (see \citet{bondy-1991}). Other classes of
graphs that satisfy Kocay's lemma include planar graphs, trees and of
course the class of all graphs. Our theorems apply to finite and
recognisable classes of graphs satisfying Kocay's Lemma. All the above
classes of graphs are recognisable as well.

Let $\mcc_n \subseteq \mcg_n$ be a finite, recognisable class of
$n$-vertex graphs satisfying Kocay's Lemma. In the rest of this paper,
we study equations obtained by applying Kocay's Lemma to $\mcc_n$. It is
useful to view this lemma as follows. Let $\mcf \coloneqq (F_1, \ldots,
F_m)$, be a sequence of graphs where $v(F_i) < n$ for each $1 \leq i
\leq m$. Let $G, G^\prime \in R \in \quotient{\mcc_n}{\sim}$, i.e.,
$G^\prime$ is a reconstruction of $G$, and since $\mcc_n$ is
recognisable, $G^\prime$ is in $\mcc_n$. Then we have
\begin{equation}\nonumber
  \sum_{H \in \ip{\mcc_n}} c(\mcf,H)s(H,G') 
  = k_{\mcf,R}, 
\end{equation}
where $k_{\mcf,R}$ is a constant that depends only on the sequence
$\mcf$ and the reconstruction class $R$, i.e., it is a reconstructible
parameter. In this expression, $c(\mcf,H)$ is constant (i.e., it is
independent of the reconstruction class) and $s(H,G')$ depends on the
isomorphism class of a particular reconstruction $G^\prime$ of $G$ under
consideration. Therefore, each application of Kocay's Lemma provides a
linear constraint on $s(H,G^\prime)$ that all reconstructions $G^\prime$
of $G$ must satisfy.

This paper is devoted to a study of systems of such linear constraints
obtained by applications of Kocay's lemma. In particular, we study the
rank of a matrix of covering numbers that we define next.

\begin{defn}\label{d:matrix}
  Let $\mcc_n$ be a finite class of graphs on $n$ vertices. Let
  $\mathfrak{F} = (\mcf_1, \mcf_2, \ldots, \mcf_l)$ be a family of
  sequences of graphs on at most $n-1$ vertices. We let
  $M_{\mathfrak{F},\ip{\mcc_n}} \in \mathbb{R}^{\card{\mathfrak{F}}
    \times \card{\ip{\mcc_n}}}$ to be a matrix whose rows are indexed by
  the sequences $\mcf_i, i=1,2,\ldots, l$ and whose columns indexed by
  the distinct isomorphism classes of graphs in $\mcc_n$. The entries of
  $M_{\mathfrak{F},\ip{\mcc_n}}$ are the covering numbers defined by
  $c(\mcf,H)$, where $\mcf \in \mathfrak{F}$ and $H \in
  \ip{\mcc_n}$.
\end{defn}

\section{On the rank of a matrix obtained from Kocay's Lemma}

\subsection{Large rank implies few non-reconstructible graphs} 
As observed earlier, for any finite class $\mcc_n$ of graphs,
$\card{\quotient{\mcc_n}{\sim}} \leq \card{\quotient{\mcc_n}{\cong}}$,
and the bigger the number of distinct reconstruction classes, the
smaller is the number of non-reconstructible graphs. The main result
of this section, Theorem \ref{t:main}, states that for any finite,
recognisable class of graphs satisfying Kocay's lemma, the number of
distinct reconstruction classes is bounded from below by the rank of the
matrix of covering numbers, for any family of sequences of graphs.

Let $\mcc_n$ be a finite, recognisable class of $n$-vertex graphs
satisfying Kocay's Lemma. Let $\mathfrak{F}$ be a finite family of
sequences of graphs on at most $n-1$ vertices. Let
$M_{\mathfrak{F},\ip{\mcc_n}}$ be the corresponding matrix of covering
numbers $c(\mcf,H)$, where $\mcf \in \mathfrak{F}$ and $H \in
\ip{\mcc_n}$ (see Definition~\ref{d:matrix}). Let $W = \{ x \in
\mathbb{R}^{\card{\quotient{\mcc_n}{\cong}}} \mid
M_{\mathfrak{F},\ip{\mcc_n}} \cdot x \equiv 0 \}$ be a subspace of the
vector space $\mathbb{R}^{\card{\quotient{\mcc_n}{\cong}}}$ over
$\mathbb{R}$. We associate with $\mcc_n$ the constant $\alpha(\mcc_n)
\coloneqq \card{\quotient{\mcc_n}{\cong}} -
\card{\quotient{\mcc_n}{\sim}}$.

\begin{lem}\label{l:dimw}
  $\mathsf{dim}(W) \geq \alpha(\mcc_n)$.
\end{lem}

\begin{proof}
  If $\alpha(\mcc_n)= 0$, the result is trivial. Otherwise, let $R_1,
  \ldots, R_s \in \quotient{\mcc_n}{\sim}$ be the non-reconstructible
  reconstruction classes in $\mcc_n$, i.e., $r_i \coloneqq
  \card{\quotient{R_i}{\cong}} > 1$ for all $i\in \{1,2,\ldots,
  s\}$. Let $R_{i,j}, j \in \{1,2,\ldots, r_i\}$ be the isomorphism
  classes in $R_i, i\in \{1,2,\ldots, s\}$. Let $G_{i,j}$ be
  representative graphs from $R_{i,j}$.

  For each $G_{i,j}$, we define a vector $w^{i,j} \in
  \mathbb{R}^{\card{\quotient{\mcc_n}{\cong}}}$, with its entries, which
  are indexed by unlabelled graphs $H \in \ip{\mcc_n}$, defined as
  follows:
  \[
  w^{i,j}(H) := s(H,G_{i,j}) - s(H,G_{i,1}), \,\text{where}\, H \in
  \ip{\mcc_n}.
  \]

  Observe that to prove the lemma it is enough to show that the vectors
  $w^{i,j}$ satisfy the following properties:
  \begin{enumerate}
  \item [(\emph{i})] for all $i \in \{1,2,\ldots, s\}$, for all $j \in
    \{1,2,\ldots, r_i\}$, $w^{i,j} \in W$; and
  \item[(\emph{ii})] the vectors in the set $U \coloneqq \{w^{i,j} \mid
    1 \leq i \leq s,\, 2 \leq j \leq r_i\}$ are non-zero and linearly
    independent, where $\card{U} = \alpha(\mcc_n)$.
  \end{enumerate}
  
  \noindent \emph{Proof of \emph{(}i\emph{)}}: Graphs $G_{i,j}$ and
  $G_{i,1}$ are reconstructions of each other, and $\mcc_n$ satisfies
  Kocay's Lemma.  Therefore, for every row $M_\mcf$ of
  $M_{\mathfrak{F},\ip{\mcc_n}}$, we have,
  \begin{eqnarray*}
    \sum_{H \in \ip{\mcc_n}} c(\mcf,H)s(H,G_{i,j}) & = & 
    \sum_{H \in \ip{\mcc_n}} c(\mcf,H)s(H,G_{i,1}) \\
    \therefore \quad M_{\mcf} \cdot w^{i,j} & = & 
    \sum_{H \in \ip{\mcc_n}} c(\mcf,H) (s(H,G_{i,j}) - s(H,G_{i,1})) = 0.
  \end{eqnarray*}
  Therefore, $M_{\mathfrak{F},\ip{\mcc_n}}\cdot w^{i,j} = 0$.

  \

  \noindent \emph{Proof of \emph{(}ii\emph{)}}: Let the vectors in $U$
  be ordered $u^1,u^2,\ldots, u^{\alpha(\mcc_n)}$ so that the
  corresponding graphs are ordered by non-decreasing numbers of small
  edges. We prove that $u^1$ is non-zero, and for each $k \in \{2,\dots,
  \alpha(\mcc_n)\}$, the vector $u^k$ is non-zero and is linearly
  independent of $u^1,u^2,\dots, u^{k-1}$, which would imply that the
  vectors in $U$ are linearly independent.

  Let $u^\ell = w^{i,j}$ for some $i \in \{1,2,\ldots, s\}$ and $j \in
  \{2,\ldots, r_i\}$. First recall that $\mcc_n$ is recognisable, $R_i
  \in \quotient{\mcc_n}{\sim}$, and $G_{i,j} \in \quotient{R_i}{\cong}$;
  therefore, $G_{i,j} \in \quotient{\mcc_n}{\cong}$. In addition,
  $G_{i,j} \ncong G_{i,1}$ since $j \geq 2$ and these two graphs belong
  to distinct isomorphism classes within the same reconstruction class
  $R_i$. Finally, the number of small edges is reconstructible, i.e.,
  $e(G_{i,j}) = e(G_{i,1})$. Therefore,
  \begin{equation}\nonumber
    u^\ell(G_{i,j}) = w^{i,j}(G_{i,j}) =
    s(G_{i,j},G_{i,j}) - s(G_{i,j},G_{i,1}) = 1 - 0 = 1.
  \end{equation}

  Now consider the vectors $u^k = w^{i^\prime,j^\prime}$ and $u^\ell =
  w^{i,j}$, where $1 \leq k < \ell$. We prove that $u^k(G_{i,j}) =
  0$. Since $k < \ell$, according to the ordering of $U$, we have
  $e(G_{i^\prime,j^\prime}) \leq e(G_{i,j})$. Since
  $G_{i^\prime,j^\prime}$ and $G_{i^\prime,1}$ are reconstructions of
  each other, we have $e(G_{i^\prime,j^\prime}) = e(G_{i^\prime,1})$.

  Now, if $e(G_{i^\prime,j^\prime}) < e(G_{i,j})$, then
  \begin{equation}\nonumber
    u^k(G_{i,j}) = w^{i^\prime,j^\prime}(G_{i,j})
    = s(G_{i,j}, G_{i^\prime,j^\prime}) - s(G_{i,j},G_{i^\prime,1}) = 0 - 0 = 0.
  \end{equation}
  On the other hand, if $e(G_{i^\prime,j^\prime}) = e(G_{i,j})$, then
  again $s(G_{i,j}, G_{i^\prime,j^\prime}) = 0$ (since $G_{i,j}$ and
  $G_{i^\prime,j^\prime}$ are non-isomorphic but have the same number of
  edges) and $s(G_{i,j},G_{i^\prime,1}) = 0$ (because $j > 1$, so
  $G_{i,j}$ and $G_{i^\prime,1}$ are non-isomorphic but have the same
  number of edges).

  Now the lemma follows from $\alpha(\mcc_n) \coloneqq
  \card{\quotient{\mcc_n}{\cong}} - \card{\quotient{\mcc_n}{\sim}} =
  \sum_{i=1}^s (r_i-1) = \card{U}$.
\end{proof}

\begin{thm} \label{t:main} Let $\mcc_n$ be a finite, recognisable class
  of $n$-vertex graphs satisfying Kocay's Lemma. Let $\mathfrak{F}$ be a
  family of sequences of graphs on at most $n-1$ vertices. If
  $M_{\mathfrak{F},\ip{\mcc_n}}$ is the corresponding matrix of covering
  numbers associated with $\mathfrak{F}$ and $\mcc_n$, then
  $\card{\quotient{\mcc_n}{\sim}} \geq
  \mathsf{rank}_\mathbb{R}(M_{\mathfrak{F},\ip{\mcc_n}})$.
\end{thm}

\begin{proof}
  Applying the Rank-Nullity Theorem, we have
  \[
  \mathsf{dim}(W) +
  \mathsf{rank}_\mathbb{R}(M_{\mathfrak{F},\ip{\mcc_n}}) =
  \card{\quotient{\mcc_n}{\cong}}.
  \]
  It follows from Lemma~\ref{l:dimw} that
  \[
  \alpha(\mcc_n) +
  \mathsf{rank}_\mathbb{R}(M_{\mathfrak{F},\ip{\mcc_n}}) \leq
  \mathsf{dim}(W) +
  \mathsf{rank}_\mathbb{R}(M_{\mathfrak{F},\ip{\mcc_n}}) =
  \card{\quotient{\mcc_n}{\cong}}.
  \]
  Now recalling the definition of $\alpha(\mcc_n)$, we have
  \[
  \card{\quotient{\mcc_n}{\cong}}-\card{\quotient{\mcc_n}{\sim}} +
  \mathsf{rank}_\mathbb{R}(M_{\mathfrak{F},\ip{\mcc_n}}) \leq
  \card{\quotient{\mcc_n}{\cong}},
  \]
  which implies that $\card{\quotient{\mcc_n}{\sim}} \geq
  \mathsf{rank}_\mathbb{R}(M)_{\mathfrak{F},\ip{\mcc_n}}$.
\end{proof}

\begin{cor} \label{c:fullrank} Under the hypotheses of Theorem
  \emph{\ref{t:main}}, if $\mathsf{rank}_\mathbb{R}(M_{\mathfrak{F},
    \mcc_n}) = \card{\quotient{\mcc_n}{\cong}}$ then every graph in
  $\mcc_n$ is reconstructible.
\end{cor}

Figure \ref{f:covers} illustrates an application of Corollary
\ref{c:fullrank} to the class of connected graphs on four vertices. We show six sequences of graphs
(indexing rows) and the corresponding covering numbers for each of the six
connected graphs on four vertices (indexing the columns). A zero in $i$-th row
and the $j$-th column (e.g., most entries in the upper triangle) indicates
that there is no way to cover the corresponding graph (indexing a
column) by graphs in the corresponding sequence (indexing the row). The
matrix has full rank, implying that connected graphs on four vertices are
reconstructible.

\begin{figure}[htb]
  \begin{center}
    \includegraphics[scale=0.5]{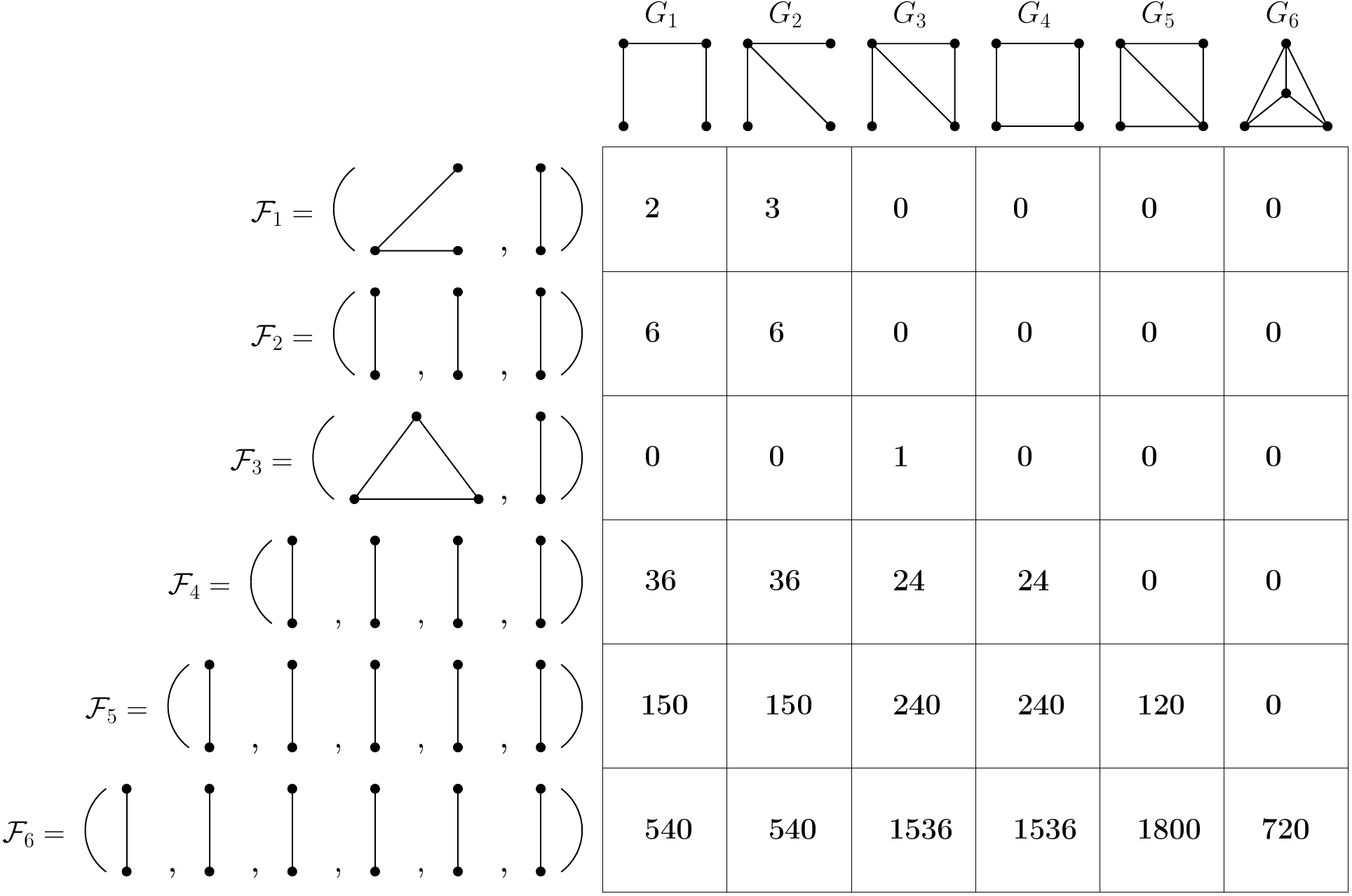}
  \end{center}
  \caption{A full-rank matrix $M$ of covering numbers $c(\mcf_i,G_j)$
    providing a proof through Corollary \ref{c:fullrank} that all
    connected graphs on four vertices are reconstructible.}
 \label{f:covers}
\end{figure}

\subsection{The existence of matrices with optimal rank}

\begin{thm}\label{t:converse}
  Let $\mcc_n$ be a recognizable class of $n$-vertex graphs satisfying
  Kocay's lemma. Then there exists a family $\mathfrak{F}$ of sequences of
  graphs with corresponding matrix of
  covering numbers  $M_{\mathfrak{F},\ip{\mcc_n}}$ such that
  $\mathsf{rank}_{\mathbb{R}}(M_{\mathfrak{F},\ip{\mcc_n}}) =
  \card{\quotient{\mcc_n}{\sim}}$.
\end{thm}

\begin{proof} Let $\mathfrak{F}$ be the family of all {\em inequivalent}
  sequences of length at most $n$ of $(n-1)$-vertex graphs. Here we
  consider two sequences $\mcf_i$ and $\mcf_j$ to be inequivalent if for
  each bijection $f$ from $\mcf_i$ to $\mcf_j$, there is at least one
  graph $F$ in $\mcf_i$ for which $f(F)$ is not isomorphic to $F$. Since
  the covering numbers for sequences of length 1 in $\mathfrak{F}$ are
  all 0, we assume that $\mathfrak{F}$ contains only sequences of length
  at least 2. Let $M_{\mathfrak{F},\ip{\mcc_n}}$ be the corresponding
  matrix of covering numbers. We show below that this choice for the
  family of sequences and its corresponding matrix of covering numbers
  satisfy the desired property.

  For a sequence $\mcf$ and a graph $G$, let $c^*(\mcf,G)$ denote the
  number of tuples $(G_1,G_2, \ldots , G_m)$ of subgraphs of $G$ with
  distinct vertex sets such that $G_i \cong F_i$, $1 \leq i \leq m$, and
  $\bigcup G_i = G$. We call such covers {\em
    non-overlapping}. Correspondingly, we have the matrix
  $M^*_{\mathfrak{F},\ip{\mcc_n}}$ of non-overlapping covering numbers.

  Now let $\mcf := (F_1,F_2,\ldots, F_\ell)$ be a sequence in
  $\mathfrak{F}$. We have the following recurrence for $c(\mcf,G)$:

  \begin{equation*}
    c(\mcf,G) = 
    \sum_{k=2}^\ell \sum_{P \in \mathcal{P}_\ell^k} 
    \sum_{\mathcal{H}:= (H_1,H_2,\ldots, H_k)} \gamma(\mathcal{H})
    c^*(\mathcal{H},G)\prod_{i=1}^{k}c(\mcf|_{P^{-1}(i)},H_i),
  \end{equation*}
  where $\mathcal{P}_\ell^k$ denotes the set of all onto functions from
  $\{1,2,\dots,\ell\}$ to $\{1,2,\dots,k\}$, and $\mcf|_{P^{-1}(i)}$ is
  the subsequence of $\mcf$ consisting of $F_j; j \in P^{-1}(i)$, and
  the innermost sum is over all inequivalent sequences $\mathcal{H}$ of
  length $k$ of graphs on $(n-1)$ vertices. This may be explained as
  follows. Each cover $(G_1,G_2, \ldots , G_\ell)$ of $G$ by $\mcf$
  naturally corresponds to a partition of $\{1,2,\dots,\ell\}$ in $k$
  blocks for some $k \in [2..\ell]$, so that $i,j$ are in the same
  partition if and only if graphs $G_i$ and $G_j$ have the same vertex
  set. We denote partitions of $\{1,2,\dots,\ell\}$ in $k$ blocks by
  onto maps $P$ from $\{1,2,\dots,\ell\}$ to $\{1,2,\dots,k\}$ so that
  the inverse image $P^{-1}(i)$ denotes the $i$-th block. For the $i$-th
  block $P^{-1}(i)$ of an onto map $P$, the union of graphs $G_j; j \in
  P^{-1}(i)$ is a graph $H_i$ on $n-1$ vertices. We denote the
  subsequence of $\mcf$ with indices $j \in P^{-1}(i)$ by
  $\mcf|_{P^{-1}(i)}$. Now the cover of $G$ by the sequence
  $\mathcal{H}:= (H_1,H_2,\ldots, H_k)$ is non-overlapping, and each
  $H_i$ may be covered by $F_j; j \in P^{-1}(i)$ in
  $c(\mcf|_{P^{-1}(i)},H_i)$ ways. We do not need to consider the
  trivial partition of $\{1,2,\dots,\ell\}$ into a single block, because
  there is no cover $(G_1,G_2, \ldots , G_\ell)$ of $G$ by $\mcf$ such
  that all $G_i$ have the same vertex set. In other words, the above
  formula computes $c(\mcf,G)$ by partitioning the coverings according
  to $k$, $P$, and $\mathcal{H}$, and then counting the number of
  coverings in each block of the partition. Since in the formula we use
  onto functions instead of partitions, the same block of coverings
  under this partition may be counted more than once, and therefore
  there is factor $\gamma(\mathcal{H})$ in the formula. If sequence
  $\mathcal{H}$ contains $k_1$ copies of a graph $\Gamma_1$, $k_2$
  copies of a graph $\Gamma_2$, and so on, where $\Gamma_i$ are mutually
  non-isomorphic graphs, then $\gamma(\mathcal{H}) = \left( \prod_i k_i!
  \right)^{-1}$.
  
  Now we rearrange the terms and write
  \begin{equation*}
    c^*(\mcf,G) = c(\mcf,G)
    - \sum_{k=2}^{\ell-1} \sum_{P \in \mathcal{P}_\ell^k} 
    \sum_{\mathcal{H}:= (H_1,H_2,\ldots, H_k)} \gamma(\mathcal{H})
    c^*(\mathcal{H},G)\prod_{i=1}^{k}c(\mcf|_{P^{-1}(i)},H_i).
  \end{equation*}
  Thus we have expressed the non-overlapping covering numbers for a
  sequence of length $\ell$ of graphs in terms of the non-overlapping
  covering numbers for sequences of length at most $\ell-1$. In the
  above equation, $c(\mcf|_{P^{-1}(i)},H_i)$ are constants independent
  of $G$. Also, if $\ell=2$, we have $c^*(\mcf,G) =
  c(\mcf,G)$. Therefore, by repeatedly applying the above equation to
  terms containing non-overlapping covering numbers, we eventually
  obtain
  \begin{equation*}
    c^*(\mcf,G) = 
    \sum_{\mathcal{F^\prime}} 
    \beta_\mathcal{F}(\mathcal{F^\prime})c(\mathcal{F^\prime},G).
  \end{equation*}
  We have written the coefficients as
  $\beta_\mathcal{F}(\mathcal{F^\prime})$ to emphasize that they arise
  from factors $c(\mcf|_{P^{-1}(i)},H_i)$ and $\gamma(\mathcal{H})$ that
  do not depend on $G$. That is, the linear dependence of the
  non-overlapping covering numbers on the covering numbers is the same
  for all graphs (but of course depends on $\mcf$). Therefore, we can
  write
  \begin{equation*}
    c^*(\mcf,\cdot) = 
    \sum_{\mathcal{F^\prime}} 
    \beta_\mathcal{F}(\mathcal{F^\prime})c(\mathcal{F^\prime},\cdot).
  \end{equation*}
  In this manner we have shown that the rows of
  $M^*_{\mathfrak{F},\ip{\mcc_n}}$ are in the span of the rows of
  $M_{\mathfrak{F},\ip{\mcc_n}}$. Therefore, we have
  \begin{equation*}\label{e:rank2}
    \mathsf{rank}_{\mathbb{R}}(M^*_{\mathfrak{F},\ip{\mcc_n}}) \leq
    \mathsf{rank}_{\mathbb{R}}(M_{\mathfrak{F},\ip{\mcc_n}}).
  \end{equation*}

  To show that the rank of $M^*_{\mathfrak{F},\ip{\mcc_n}}$ is
  $\card{\quotient{\mcc_n}{\sim}}$, we construct a square submatrix $K$
  of $M^*_{\mathfrak{F},\ip{\mcc_n}}$ as follows. Let $\{R_i, i =
  1,2,\dots\} \coloneqq \quotient{\mcc_n}{\sim}$. First, for each
  reconstruction class $R_i, i = 1,2, \dots$, we choose one
  reconstruction $G_i$ arbitrarily from $\quotient{R_i}{\cong}$. For
  each $i = 1,2,\dots$, we keep the row indexed by the sequence (say
  $\mcf_i$) that is equivalent to the sequence $(G_i - v, v \in
  V(G_i))$, where the vertices of $G_i$ may be ordered arbitrarily, and
  we keep the column indexed by $G_i$. We delete all other rows and
  columns of $M^*_{\mathfrak{F},\ip{\mcc_n}}$. We show that $K$ has full
  rank, which will imply that
  $\mathsf{rank}_{\mathbb{R}}(M^*_{\mathfrak{F},\ip{\mcc_n}}) \geq
  \mathsf{rank}_{\mathbb{R}}(K) = \card{\quotient{\mcc_n}{\sim}}$.

  We define a partial order $\leq$ on $\quotient{\mcc_n}{\sim}$ so that
  $R_i \leq R_j$ if there exists a bijection $f$ from $V(G_i)$ to
  $V(G_j)$ such that for each $v$ in $V(G_i)$, the graph $G_i -v$ is
  isomorphic to a subgraph of $G_j - f(v)$.

  First we verify that the above relation $\leq$ is a partial order on
  $\quotient{\mcc_n}{\sim}$. The reflexivity and the transitivity are
  straightforward to verify. We now verify antisymmetry. Let $f$ be a
  bijection as in the above paragraph. Therefore, for each $v \in
  V(G_i)$, we have $e(G_i - v) \leq e(G_j - f(v))$. Let $g$ be a similar
  bijection from $V(G_j)$ to $V(G_i)$. Therefore, the bijective
  composition $g\circ f$ from $V(G_i)$ to $V(G_i)$ is such that for all
  $v$ in $V(G_i)$, we have $G_i - v$ is isomorphic to a subgraph of $G_i
  - (g\circ f)(v)$, implying that $e(G_i - v) \leq e(G_j - f(v)) \leq
  e(G_i - (g\circ f)(v))$.  Now observe that $\sum_v e(G_i - v) = \sum_v
  e(G_i - (g\circ f)(v))$, since $g\circ f$ is a bijection from $V(G_i)$
  onto itself. Therefore, we must have $e(G_i-v) = e(G_j - f(v))$ for
  all $v \in V(G_i)$, implying that $G_i - v$ and $G_j - f(v)$ are
  isomorphic for all $v \in V(G_i)$. In other words, $R_i = R_j$.

  We sort the rows and the columns of $K$ so that if $R_i < R_j$, then
  $G_j$ is to the right of $G_i$, and the row corresponding to the
  sequence $\mcf_i$ is above the row corresponding to the family
  $\mcf_j$.

  Now if $c^*(\mcf_i, G_j) > 0$ then $R_i < R_j$, therefore, the matrix
  $K$ is upper-triangular. Also, $c^*(\mcf_i, G_i) > 0$ for all
  $G_i$. Therefore, $K$ has full rank; in fact $\mathsf{rank}_{\mathbb{R}}(K)$ is
  equal $\card{\quotient{\mcc_n}{\sim}}$. Since the class $\mcc_n$ is
  recognizable and satisfies Kocay's lemma, Theorem \ref{t:main} is
  applicable. Therefore,
  \[
  \card{\quotient{\mcc_n}{\sim}} = \mathsf{rank}_{\mathbb{R}}(K) \leq
  \mathsf{rank}_{\mathbb{R}}(M^*_{\mathfrak{F},\ip{\mcc_n}}) \leq
  \mathsf{rank}_{\mathbb{R}}(M_{\mathfrak{F},\ip{\mcc_n}}) \leq
  \card{\quotient{\mcc_n}{\sim}},
  \]
  which implies the claim for our choice of $\mathfrak{F}$, and the
  corresponding matrix $M_{\mathfrak{F},\ip{\mcc_n}}$.
\end{proof}

\begin{example} 
  We provide another simple but non-trivial example in directed graphs,
  which are in general not reconstructible. Figure \ref{f:covers2}
  illustrates a matrix of covering numbers for directed graphs on 3
  vertices, with no multi-arcs or loops. Observe that there are 7
  distinct graphs in 4 reconstruction classes: $G_1$ and $G_2$ are
  reconstructible; $G_3,G_4,G_5$ belong to the same reconstruction
  class; $G_6,G_7$ belong to the same reconstruction class. The figure
  shows 4 rows of the matrix corresponding to 4 graph sequences. The
  rank of the matrix is 4, which is also the number of reconstruction
  classes. It is possible to verify that the rank cannot be improved by
  adding more sequences of graphs.
\end{example}

\begin{figure}[htb]
  \begin{center}
    \includegraphics[scale=0.5]{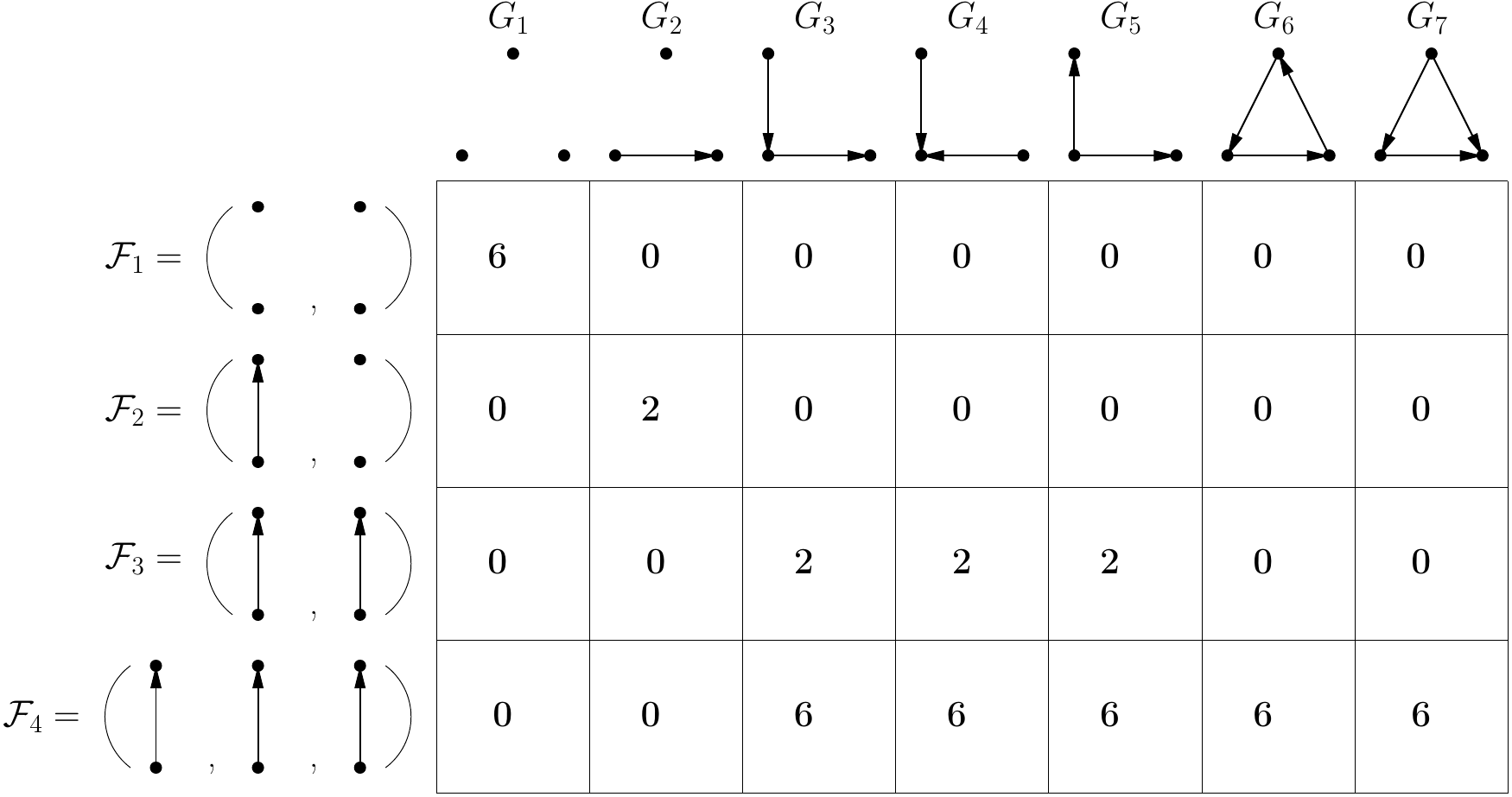}
  \end{center}
  \caption{A matrix of covering numbers for directed graphs on 3
    vertices. There are 4 reconstruction classes and the rank of the
    above matrix is also 4.}
 \label{f:covers2}
\end{figure}

\section*{Acknowledgements}

We would like to thank Hi\d{\^{e}}p H\`{a}n for useful discussions, and the anonymous referees for bringing our attention to the works of Kocay \citep{kocay1982} and Mnukhin \citep{mnukhin1992}. The first author is grateful to Yoshiharu Kohayakawa for hosting him at
Universidade de S\~{a}o Paulo, and would like to thank Orlando Lee for helpful discussions at an early stage of this work.

\bibliography{refs}	

\end{document}